\documentclass[12pt]{amsart}

\usepackage{amsmath,amsthm,amsfonts,amssymb,latexsym}
\usepackage{verbatim,graphicx,ifthen,enumitem,color,a4}

\usepackage[T1]{fontenc}
\usepackage[applemac]{inputenc}

\usepackage[svgnames]{xcolor}
\usepackage{colortbl}
\usepackage[colorlinks=true,linkcolor=azul,citecolor=azul]{hyperref}

\usepackage{mathtools}

\theoremstyle{plain}
\newtheorem{theorem}{Theorem}[section]
\newtheorem{lemma}[theorem]{Lemma}
\newtheorem{corollary}[theorem]{Corollary}
\newtheorem{proposition}[theorem]{Proposition}

\theoremstyle{definition}

\newtheorem{remark}[theorem]{Remark}

\newtheorem{definition}[theorem]{Definition}

\newtheorem{examples}[theorem]{Examples}
\numberwithin{equation}{section}

\definecolor{azul}{rgb}{0.1,0.6,0.86}
\definecolor{bluee}{rgb}{0,0.33,0.55}
\definecolor{naranja}{RGB}{249,153,96}

\def\noi{\noindent}

\def\bdem{\begin{proof}}
\def\edem{\end{proof}}

\def\bm{\left(\begin{array}}
\def\em{\end{array}\right)}
\def\ben{\begin{enumerate}}
\def\een{\end{enumerate}}
\def\barr{\begin{array}}
\def\earr{\end{array}}
 \def\bit{\begin{itemize}}
\def\eit{\end{itemize}}
\def\beq{\begin{equation}}
\def\eeq{\end{equation}}
\def\bdes{\begin{description}}
\def\edes{\end{description}}

\def\eps{\varepsilon}
\def\fii{\varphi }
\def\la{\lambda}

\def\w{\omega}

\def\N{\mathbb{N}}

\def\R{\mathbb{R}}

\def\cR{{\mathcal R}}

\newcommand{\peso}[1]{ \quad \mbox{  #1 } \quad }

\newcommand{\norm}[1]{\left\lVert#1\right\rVert}
\newcommand{\parenthesis}[1]{\left( #1 \right)}
\newcommand{\sqbracket}[1]{\left[ #1 \right]}

\newcommand{\Rn}{{\mathbb{R}^n}}

\def\llouno{{\Lambda^1(\w)}}
\def\lloinuno{{\Lambda^{1,\infty}(\w)}}
\newcommand{\bs}[1]{[#1]_{B^*_\infty}}

\newcommand{\bunor}[1]{[#1]_{B^\mathcal{R}_1}}

\def\bdem{\begin{proof}}
\def\edem{\end{proof}}

\newcommand{\bsq}[2]{[#1]_{B^*_{#2}}}

\begin{document}

\baselineskip=17pt    

\title{From weak type weighted inequality to pointwise estimate for the decreasing rearrangement} 
\author{Elona Agora, Jorge Antezana, Sergi Baena-Miret, Mar\'{\i}a J. Carro} 

\address{E. Agora, Instituto Argentino de Matem\'atica ``Alberto P. Calder\'on'', 
1083 Buenos Aires, Argentina.}
\email{elona.agora@gmail.com}

\address{J. Antezana, Department of Mathematics, 
Faculty of Exact Sciences, 
National University of La Plata, 
1900 La Plata, Argentina,  \and
Instituto Argentino de Matem\'atica ``Alberto P. Calder\'on'', 
1083 Buenos Aires, Argentina.}
\email{antezana@mate.unlp.edu.ar}

\address{S. Baena-Miret, Department of Mathematics and Informatic, University of Barcelona, Barcelona, Spain.}
\email{sergibaena@ub.edu}

\address{M. J. Carro, Department of Applied Mathematics and Mathematical Analysis,
Complutense University of Madrid, Madrid, Spain.}
\email{mjcarro@ucm.es}

\thanks{This work has been partially supported by Grants MTM2016-75196-P (MINECO / FEDER, UE), PIP-152 (CONICET), PICT 2015-1505 (ANPCYT), and 11X829 (UNLP)}

\begin{abstract} 
We shall prove pointwise estimates for the decreasing rearrangement of $Tf$, where $T$ covers a wide range of interesting operators in Harmonic Analysis such as operators satisfying a Fefferman-Stein inequality, the Bochner-Riesz operator, rough operators, sparse operators, Fourier multipliers, etc. In particular, our main estimate is of the form 
$$
(Tf)^*(t) \le C\parenthesis{\frac 1t\int_0^tf^*(s)\,ds + \int_t^\infty \left( 1 + \log\frac st \right)^{- 1}\varphi\parenthesis{1 + \log\frac st}f^*(s)\,\frac{ds}s},
$$
where $\varphi$ is determined by the Muckenhoupt $A_p$-weight norm behaviour of the operator. 
\end{abstract}

\date{\today}

\maketitle

\pagestyle{headings}\pagenumbering{arabic}\thispagestyle{plain}

\markboth{From weak to pointwise}
{E. Agora, J. Antezana, S. Baena-Miret and M.J. Carro}

\tableofcontents

\section{Introduction} 

There are many interesting operators in Harmonic Analysis satisfying that, for every $u\in A_1$, 
\begin{equation}\label{bound}
T:L^1(u)\longrightarrow  L^{1, \infty}(u)
\end{equation}
is bounded with constant less than or equal to $C||u||_{A_1}^k$, $k \in \N$, where we recall that $A_1$ is the class of Muckenhoupt weights such that
\begin{equation}\label{mmhh}
Mu(x)\le C u(x), \quad \mbox{a.e. } x \in \mathbb R^n.
\end{equation}
Here, $M$ is the Hardy-Littlewood maximal operator defined by 
$$
Mf(x)=\sup_{Q\ni x}\frac{1}{|Q|}\int_Q |f(y)|dy, \qquad f \in L^1_{\text{loc}}(\Rn), \hspace{2mm} x\in \Rn,
$$
where the supremum is taken over all cubes $Q\subseteq\mathbb R^n$ containing $x$, and $||u||_{A_1}$ is the infimum of all constants $C$ in \eqref{mmhh}. In particular, it is known \cite{m:m} that $M$ satisfies \eqref{bound} with $||M||_{L^1(u)  \longrightarrow  L^{1, \infty}(u)} \le c_n ||u||_{A_1}$. 

Other examples of such operators are the Hilbert transform, the Riesz transform and,  more generally, any Calder\'on-Zygmund operator (see for instance \cite{lop:lop}). Also, sparse operators \cite{hp:hp}, the Bochner-Riesz operator $B_\lambda$ at the critical index (that is, $\lambda=\frac{n-1}2$) \cite{lprr:lprr, v:v}, among many others, are known to satisfy  condition \eqref{bound}. 

The main purpose of this paper is to prove pointwise estimates for the decreasing rearrangement of $Tf$ with respect to the Lebesgue measure. In particular, our first main result is the following: 

\begin{theorem}\label{Teo-main1}
Let $T$ be a sublinear operator such that, for every $u\in A_1$ and some  $k \in \mathbb{N}$, $T$ satisfies \eqref{bound} with constant less than or equal to $C||u||_{A_1}^k$. 
Then,  for every $t>0$ and every measurable function $f$, 
\begin{equation*}%\label{estimateTeo3}
(Tf)^*(t)\leq C\parenthesis{\frac 1t \int_0^t f^*(s) \, ds + \int_t^\infty \left(1 + \log\frac st\right)^{k - 1} f^*(s) \, \frac{ds}s}.     
\end{equation*}

\end{theorem}

This pointwise estimate is very interesting since obviously it has, as a consequence, boundedness properties of such operators on rearrangement invariant spaces (for more details, we refer to \cite{bs:bs}). In particular, it allows us to characterize the weights $\w$ for which they are bounded on weighted Lorentz spaces $\Lambda^{p,q}(\w)$ \cite{crs:crs}. 

\medskip

Now, on some occasions, the behaviour of the constant has been improved from, let us say, $||u||_{A_1}^{1+\varepsilon}$, $\varepsilon > 0$, to an expression of the form $ \varphi(||u||_{A_1})$, where $\varphi$ is not a power function. This is, for example, the case when $T$ is a Calder\'on-Zygmund operator,  where the best function $\varphi$ known up to now  is $\varphi(t)=t (1+\log^+t)(1+\log^+\log^+ t)$ \cite{lop:lop}. 

\medskip

\noindent In order to cover this important class of operators we shall, in fact, prove a more general version of Theorem \ref{Teo-main1} (see Definition \ref{admissible} for the concept of
admissible function).

\begin{theorem}\label{Teo-main1-bis}
Let $T$ be a sublinear operator such that, for every $u\in A_1$, 
$$ 
T:L^1(u)\longrightarrow  L^{1, \infty}(u),
$$
with constant less than or equal to $\varphi(||u||_{A_1})$, where  $\fii$ is an admissible  function. Then, for every $t>0$ and every measurable function $f$, 
\begin{equation}\label{dec}
(Tf)^*(t)\leq C\parenthesis{ 
 \frac 1t \int_0^t f^*(s) \, ds + \int_t^\infty \left(1 + \log\frac st\right)^{-1}\fii\left(1 + \log\frac st\right)  f^*(s) \, \frac{ds}s}.    
\end{equation}
\end{theorem}

To prove inequality \eqref{dec}, our starting point was the following result (see \cite{bc:bc}). See Section \ref{lemmas} for the definitions of the classes of weights $B_1^{\mathcal R}$ and $B_{\infty}^{*}$. 

\begin{theorem}\label{main1} 
Let $T$ be an operator satisfying that,  for  every $u \in A_1$, 
$$
T:L^{1}(u)  \rightarrow L^{1,  \infty}(u)
$$ is bounded, with constant less than or equal to $\varphi(||u||_{A_1})$, where $\varphi$ is an increasing function in $[1,\infty)$. Then, for every $\w \in B_1^{\mathcal R}\cap B_{\infty}^{*}$, 
 $$
 T:\Lambda^1(\w)  \rightarrow \Lambda^{1, \infty}(\w)
 $$
  is bounded with constant less than or equal to $C_1[\w]_{B_1^{\mathcal \cR}}\varphi\left(C_2[\w]_{B_\infty^*} \right)$ for some positive constants $C_1$, $C_2$ independent of $\w$.   
\end{theorem}

Let us now see an easy proposition which helps to motivate what follows: 

\begin{proposition}
Let $T$ be an operator satisfying that, for every $\w\in B_1^{\cR}$,
\begin{equation}\label{uuu1}
T:\llouno\to\lloinuno
\end{equation}
is bounded with constant less than or equal to $\fii(\bunor{\w})$. Then, for every $t>0$, 
\begin{equation}\label{uuu}
(Tf)^*(t)\leq \frac{\fii(1) }{t}\int_0^t f^*(s)\, ds. 
\end{equation}
\end{proposition}

\bdem
The hypothesis implies that, for every $ \w \in B_1^\cR$ and every $t>0$, 
\begin{equation*} 
 (Tf)^*(t) W(t) \leq \fii(\bunor{\w}) \int_0^\infty f^*(s) \, \w(s)\, ds.
\end{equation*}
In particular, since $\w=\chi_{[0,t]}\in B_1^\cR$ and $\bunor{\w}=1$,  we get that
$$
(Tf)^*(t) \le  \frac{ \fii(1)}{t}\int_0^{t} f^*(s) \, ds.
$$
\edem

\begin{remark} \noindent i) If an operator $T$ satisfies \eqref{uuu}, we have that \eqref{uuu1} holds with constant less than or equal to $C[\w]_{B_1^{\cR}}$ and hence we can conclude that under the hypothesis of the previous theorem, $||T|| \leq \tilde C\min ( [\w]_{B_1^{\cR}}, \fii([\w]_{B_1^{\cR}}) )$, for some positive constant $\tilde C$ independent of $\w$. 

\noindent
ii) We also observe that the operator $T$ plays no role and hence the same can be formulated for couples of functions $(f, g)$ in the following sense: 
$$
||g||_{\Lambda^{1, \infty}(\w)} \le \fii([\w]_{B_1^{\cR}}) ||f||_{\Lambda^{1}(\w)}, \qquad \forall \w\in B_1^{\cR}, 
$$
implies that 
$$
g^*(t) \le \fii(1) f^{**}(t), \qquad \forall t>0.
$$
\end{remark}

Taking into account Theorem \ref{main1}, our next goal was to include the hypothesis $\w\in B^*_\infty$ in its statement:

\begin{theorem}\label{Teo-1}
Let $T$ be a sublinear operator and let $\varphi$ be an admissible function. Then, for every $\w\in B_1^{\cR}\cap B^*_\infty$,
$$
T:\llouno\to\lloinuno
$$
is bounded with norm less than or equal to $C\bunor{\w}\ \fii(\bs{\w})$ if and only if,  for every $t>0$, and every measurable function $f$, 
\begin{align*}
(Tf)^*(t)&\leq C\left(\frac 1t \int_0^t f^*(s) \, ds + \int_t^\infty \left(1 + \log\frac st\right)^{-1}\fii\left(1 + \log\frac st\right)  f^*(s) \, \frac{ds}s\right).   \nonumber
\end{align*}
\end{theorem}

\medskip
Besides, from the proof of the above theorem, we obtain  the following result. 

\begin{corollary} \label{fip} Let $T$ be a sublinear operator such that, for some admissible function $\varphi$ and any $1\le p<\infty$,
$$
||Tf||_{L^{p, \infty}} \le C \varphi(p) ||f||_{L^{p,1}}, 
$$
with $C$ independent of $p$, then \eqref{dec} holds. 
\end{corollary}

As a consequence, we have the following result: 

\begin{theorem}  \label{fip2} If 
$$
T:L^1 \longrightarrow L^{1, \infty}
$$
and, there exists $p_0>1$ so that 
$$
T:L^{p_0, 1}(u) \longrightarrow L^{p_0,\infty}(u),
$$
with constant $C \varphi(||u||_{A_{p_0}})$, for some admissible function $\varphi$, then \eqref{dec} holds. 
\end{theorem}

Finally, our technique can be  also applied to operators for which condition \eqref{bound}
is changed by a weaker one, as the following result shows.

\begin{theorem}\label{Teo-main2-bis} 
Let $T$ be a sublinear operator such that, for some $0<\alpha\le 1$, some $p_0\ge 1$ and every $u\in A_1$,
$$
T:L^{p_0, 1}(u^\alpha)\longrightarrow  L^{p_0, \infty}(u^\alpha), 
$$
with constant less than or equal to $\varphi(||u||_{A_1})$, where $\varphi$ is an admissible function. Then, for every  $t>0$ and every measurable  function $f$, 
$$
(Tf)^*(t)\leq C \left(\frac 1{t^{\frac 1{p_0}}}\int_0^tf(s)\,\frac{ds}{s^{1 - \frac 1{p_0}}}+ \frac1{t^{\frac {1-\alpha}{p_0}}}\int_t^\infty \tilde\fii\left(1+\log \frac st\right) f(s)\,\frac{ds}{s^{1 - \frac {1-\alpha}{p_0}}}\right),
$$

with $$\tilde{\varphi}(x) = \left\{\begin{array}{ll} \varphi(x^{\frac{p_0}\alpha}), & 0 < \alpha < 1, \vspace{3mm} \\ x^{-1}\varphi(x), & \alpha = 1. \end{array}\right.$$
\end{theorem}

The paper is organized as follows. In Section \ref{technic},  we present previous results, the necessary definitions and some technical lemmas which shall be used later on. Section  \ref{main}  contains the proofs of our  results and Section \ref{operators_bound} will be devoted to obtain pointwise estimates for the decreasing rearrangement of $Tf$ with respect to the Lebesgue measure for $T$  being operators satisfying a Fefferman-Stein inequality, the Bochner-Riesz operator, a rough operator,  a sparse operator or a Fourier multiplier. 
As usual, we write $A \lesssim B$ if there exists a positive constant $C>0$, independent of $A$ and $B$, such that $A\leq C B$. If $A\lesssim B$ and $B\lesssim A$, then we write $A\approx B$. 

\section{Definitions, previous results and lemmas}\label{technic}\label{lemmas}

\subsection{Admissible functions}

\begin{definition} \label{admissible} A function $\varphi:[1, \infty]\to [1, \infty]$ is called admissible if satisfies 
the following conditions:
\begin{enumerate}
\item[a)] $\fii(1) = 1$  and it is log-concave, that is 
$$
\theta\log \varphi(x)+(1-\theta) \log \varphi(y) \le \log \varphi(\theta x+(1-\theta)y), \qquad \forall x, y \geq 1, \ 0\le \theta\le 1, 
$$
\item[b)] and there exist $\gamma,\beta>0$ such that for every $x \geq 1$,
\begin{equation}\label{la espada y la pared}
\frac\gamma x\leq \frac{\fii'(x)}{\fii(x)}\leq \frac\beta x.
\end{equation}
\end{enumerate}
\end{definition}

\noindent Observe that \eqref{la espada y la pared}  implies that $\fii$ is increasing, as well as that 
$$
x^\gamma\leq \fii(x)\leq x^\beta. 
$$

\noindent Besides, since for every $C > 0$,
\begin{equation*}
\begin{split}
 \log\fii(C x) &=\int_1^x(\log \fii)'(s)\,ds+\int_x^{Cx}(\log \fii)'(s)ds \leq \log \fii(x) + \beta \log_+ C,   
\end{split}
\end{equation*}

\noindent it also holds that \begin{equation}\label{por la logconcavidad2}
    \varphi(Cx) \leq \max\{1, C^\beta\} \varphi(x).
\end{equation}

From now on, the function $\varphi$ will be an admissible function. 
\begin{examples}

\noindent
i)  Given $\gamma > 0$ and $\beta\geq 0$, the function $\fii(x)=x^\gamma(1+\log x)^\beta$ is admissible. 

\noindent
ii)  Given $n\in\N$, define
$$
\log_{(n)}x=\begin{cases}
1+\log x, &\mbox{if $n=1$},\\
1+\log\big(\log_{(n-1)} x\big), &\mbox{if $n>1$}.
\end{cases}
$$
Using this notation, if $\gamma>0$ and $\beta_1,\ldots,\beta_n\geq 0$, the function 
$$
\fii(x)=x^\gamma \prod_{k=1}^n \big(\log_{(k)} x\big)^{\beta_k}
$$
is also admissible.
\end{examples}

The next lemmas are simple computations for admissible functions which shall be fundamental in the proof of our main results. 

\begin{lemma}\label{comparables}
Let $0 < q \leq \infty$. If $r\geq 1$ then 
\begin{equation*}
\left\{ \begin{array}{lc}
    \displaystyle \int_1^r \fii\big(1+\log s\big)  \, \frac{ds}{s^{1 - \frac 1q}} \approx \fii\left(1+\log r\right)r^{\frac 1q} - 1, & q < \infty, \vspace{3mm} \\ \displaystyle
    \int_1^r \big(1+\log s\big)^{-1} \,\fii\big(1+\log s\big)  \, \frac{ds}s \approx \fii\left(1+\log r\right)-1,\, & q = \infty.
\end{array} \right. 
\end{equation*}
\end{lemma}

\begin{lemma}\label{comparables 2}
Let $1 < q \leq \infty$. There exists some $\la = \la(\varphi, q) > 1$ such that for every $t>0$ and $r \geq \la t$,
$$
\left\{\begin{array}{lc}
    \displaystyle \int_r^{\infty} \fii\left(1+\log \frac st\right)  \frac{ds}{s^{2-\frac 1q}}\approx  \fii\left(1+\log \frac rt\right) \frac{1}{r^{1-\frac 1q}}, & q <\infty, \vspace{3mm} \\
    \displaystyle \int_r^{\infty} \left(1+\log \frac st\right) ^{-1} \fii\left(1+\log \frac st\right)  \frac{ds}{s^2}\approx  \left(1+\log \frac rt\right)^{-1} \,\fii\left(1+\log \frac rt\right) \frac 1r,\, & q = \infty.
\end{array}\right.
$$
\end{lemma}
\begin{proof}
Consider the function 
$$
g_q(r)= \left\{\begin{array}{lc}
    \displaystyle -\fii\left(1+\log \frac rt\right)\  \frac1{r^{1-\frac 1q}} , & q < \infty, \vspace{3mm}\\
    \displaystyle - \left(1+\log \frac rt\right)^{-1}\fii\left(1+\log \frac rt\right)\  \frac1r,\, & q = \infty.
\end{array}\right.
$$
Then, straightforward computations show that there exists some $\la>1$ depending only on $\fii$ and on $q$ such that for every $r \geq \la t$, $$
g_q'(r)\approx \left\{\begin{array}{lc}
    \displaystyle \fii\left(1+\log \frac rt\right)\ \frac 1{r^{2-\frac 1q}}, & q < \infty, \vspace{3mm} \\
    \displaystyle \left(1+\log \frac rt\right)^{- 1} \fii\left(1+\log \frac rt\right)\ \frac 1{r^2},\, & q = \infty.
\end{array}\right.
$$
Thus, since $\lim_{r\rightarrow \infty} g_q(r) = 0$, the result follows.
\end{proof}

\begin{lemma}\label{L infimo log log}
Given $x\in\R$ and $0 < \mu \leq 1$. Then
$$
\inf_{y \in (0, \mu]}\fii(y^{-1})e^{y x}\lesssim 
\begin{cases}
e^{\mu x}, & \mbox{if $x\leq 0$},\\
\fii\left(\frac{1+x}\mu\right), & \mbox{if $x >0$}.
\end{cases}
$$
\end{lemma}
\bdem
If $x \le 0$, the infimum is attained at $y = \mu$, and if $x>0$, we take $y = \mu/(1+x)$. 
\edem

\begin{lemma}\label{L sup log log}
For every $y\geq 1$, 
$$
\sup_{x \in [1,\infty)}\fii(x)e^{-x/y }\lesssim \, \fii(y).
$$
\end{lemma}
\bdem
\noindent By means of \eqref{por la logconcavidad2}, 
$$\fii(x)e^{-x/y} \leq \max\left\{1, \left( \frac xy \right)^\beta e^{-x/y}\right\}\varphi(y) \leq \max\left\{1, \beta^\beta e^{-\beta}\right\}\varphi(y). 
$$
\edem

\medskip

\subsection{Calder\'on type operators}

\begin{definition}
\noindent Let $1 \leq q_1, q_2 \leq \infty$, and let  $\fii$ be an admissible function. Then, for every positive and  real valued  measurable function $f$,  we define 
\begin{equation}\label{P_Q_Definition}
\begin{split}
    P_{q_1}f(t)&:=  \frac 1{t^{\frac 1{q_1}}}\int_0^tf(s)\,\frac{ds}{s^{1 - \frac 1{q_1}}},\\
Q_{q_2,\fii}f(t)&:= \left\{\begin{array}{lc}
    \displaystyle \frac1{t^{\frac 1{q_2}}}\int_t^\infty \fii\left(1+\log \frac st\right) f(s)\,\frac{ds}{s^{1 - \frac 1{q_2}}}, & q_2 < \infty, \vspace{2mm} \\ \displaystyle \int_t^\infty \left( 1 + \log\frac st \right)^{- 1} \fii\left(1+\log \frac st\right) f(s)\,\frac{ds}s, 
     & q_2 = \infty,
\end{array} \right.
\end{split}
\end{equation}
and
$$
S_{q_1,q_2,\fii}f(t) := P_{q_1}f(t)+Q_{q_2,\fii}f(t).
$$
\end{definition}

\medskip

\noi In particular, if  $q_1 = 1$, $q_2=\infty$, and $\fii(x) = x$, we recover the Calder\'on operator  \cite{bs:bs}
$$
Sf(t) := Pf(t) + Qf(t),  
$$ 
where $P$ and $Q$ are respectively the Hardy operator and its conjugate
$$
Pf(t) = \frac 1t\int_0^tf(s)\,ds, \qquad Qf(t) =\int_t^\infty f(s)\,\frac{ds}s. 
$$
We observe that, in general,
\begin{equation}\label{stst}
S_{q_1, q_2, {\varphi}}f(t) =\int_0^1f(st)\,\frac{ds}{s^{1 - \frac 1{q_1}}} + \left\{\begin{array}{lc}
        \displaystyle \int_1^\infty {\varphi}\left( 1 + \log s \right) f(st)\,\frac{ds}{s^{1 - \frac {1}{q_2}}}, & q_2 < \infty, \vspace{3mm} \\ \displaystyle
        \int_1^\infty \left( 1 + \log s \right)^{-1}{\varphi}\left( 1 + \log s \right) f(st)\,\frac{ds}s,\, & q_2 = \infty.
    \end{array}\right. 
\end{equation}

\medskip

For every measurable function $f$, let $f^*$ be its decreasing rearrangement defined by 
$$
f ^*(t):=\inf\{s>0:\lambda_f(s)\leq t\}, \qquad \lambda_f(t):= |\{|f|>t\}|, \qquad t > 0,
$$
and $f^{**}$ the maximal function of $f$ defined by $f^{**}(t) = P(f^*)(t)$, $t > 0$. For further information about these notions and related topics we refer to \cite{bs:bs}.

\begin{lemma}\label{GeneralizedCalderonProperty} 
Let $1 \leq q_1, q_2 \le \infty$. For every measurable function $f$,
 $$
 S_{q_1,q_2,\fii}(f^*)^{**}(t) = S_{q_1,q_2,\fii}(f^{**})(t), \qquad t > 0.
 $$
\end{lemma}

\begin{proof}  By \eqref{stst}, clearly,  $S_{q_1,q_2,\fii}(f^*)$ is a decreasing function. Then,  it holds that    
$$
S_{q_1,q_2,\fii}(f^*)^{**}(t)= P \big( S_{q_1,q_2,\fii}(f^*) )(t), \qquad t > 0, 
$$  
and the result follows immediately by the Fubini's theorem. 
\end{proof}

\medskip

\subsection{Lorentz spaces and some classes of weights}

Let $0 < p < \infty$, and $0 < q \leq \infty$. Let $\w$ be a positive locally integrable function defined on $(0, \infty)$, and define $W(t)=\int_0^t \w(r)\,dr$, $t > 0$. The weighted Lorentz space  $\Lambda^{p,q}(\w)$ is defined by the condition $||f||_{\Lambda^{p,q}(\w)} < \infty$ where
$$
||f||_{\Lambda^{p,q}(\w)}= \left\{ \begin{array}{cc}
     \displaystyle \bigg( \int_0^\infty f^*(s)^q W(s)^{\frac qp -1}\w(s) ds \bigg)^{\frac 1q}, & q < \infty, \\ \sup_{t>0} f^*(t) W(t)^{\frac 1p },
     & q = \infty.
\end{array} \right. 
$$
 
For further information about these notions and related topics we refer to \cite{bs:bs, crs:crs, l:l1}.  \medskip

\begin{definition}  The following classes of weights have appeared in the literature concerning the boundedness of Hardy and conjugate Hardy type operators on the class of monotone decreasing functions on $L^{p,q}(\w)$, denoted by $L^{p,q}_{\text{dec}}(\w)$, $0 < p < \infty$, $0 < q \leq \infty$.

\medskip

\noindent
a) $B_p^\mathcal R$ class: Concerning the Hardy operator $P$, we have \cite{cgs:cgs, cs:cs} that for $0 < p \leq 1$,
$$
P:L^p_{\text{dec}}(\w) \longrightarrow  L^{p, \infty}(\w) \quad \iff \quad \w\in B_p^\mathcal R, 
$$
where $\w\in B_p^\mathcal R$ is defined by 
\begin{equation*} 
[\w]_{B_p^\mathcal R}=\sup_{0 < r \le t <\infty} \dfrac{rW(t)^{\frac 1p}}{tW(r)^{\frac 1p}}<\infty.
\end{equation*}
In this paper, we extend this definition for the whole range $0 < p < \infty$. In fact, $B_p^R$ can be considered to be the `restricted' class of the well known $B_p$ class \cite{am:am, n:n} and it is easy to see that all weight in $B_p$ is $p$ quasiconcave, that is $B_p \subset B_p^{\mathcal R}$.

\noindent 
b) $B_q^*$ class: Concerning the generalized conjugate Hardy-type operator \cite{l:l3, n:n} (see \eqref{P_Q_Definition}), $$Q_{q_2}(t) := Q_{q_2,1}f(t) = \frac 1{t^\frac 1{q_2}}\int_t^\infty f(s)\,\frac{ds}{s^{1 - \frac 1{q_2}}}, \qquad t > 0,$$ on $L^p_{\text{dec}}(\w)$, it holds that for $0 < q_2 < \infty$, $$Q_{q_2}:L^p_{\text{dec}}(\w) \longrightarrow L^p(\w) \quad \iff \quad w \in B_{\frac{q_2}p}^*,$$ where, for $0 < q < \infty$, $\w\in B_q^*$ if
\begin{align*}
[\w]_{B_q^{*}} := \sup_{t > 0}\frac{1}{W(t)}\int_0^t \left( \frac ts\right)^{\frac 1q}\w(s)\,ds < \infty.
\end{align*}
\noindent In particular, if $\w\in B^*_q$, we have that,  for every $0 < r < t$, 
$$
W(r)\parenthesis{\frac tr}^{\frac 1q}\le \int_0^t \parenthesis{\frac ts}^{\frac 1q}\w(s) ds\leq [\w]_{B^*_q} W(t), 
$$
and therefore, 
\begin{equation}\label{eq 1 de Bp}
\frac{W(r)}{W(t)} \le [\w]_{B^*_q} \parenthesis{\frac rt}^{\frac 1q}. 
\end{equation}

\noindent
c) $B_\infty^*$ class: Concerning the adjoint of the Hardy operator $Q$, we have \cite{a:a} that for every $p > 0$,  $$
Q:L^{p}_{\text{dec}}(\w) \longrightarrow  L^{p}(\w) \iff \w\in B_\infty^*, 
$$ where $\w\in B_\infty^*$ is defined  by 
\begin{align*}
\bs{\w}= \sup_{t > 0} \frac 1{W(t)}\int_0^t\frac{W(s)}s \,ds < \infty.
\end{align*}

\end{definition}
\noindent Hence, if $\w\in B^*_\infty$, we have that,  for every $0 < r < t$, 
$$
W(r) \log \parenthesis{\frac tr} \le \int_0^t \log \parenthesis{\frac ts} \w(s) ds = \int_0^t\frac{W(s)}s \,ds\leq \bs{\w}  W(t), 
$$
and therefore, if we define for $\la\in(0,\infty)$, 
$$
\overline{W}(\la):=\sup_{t>0} \frac{W(\la t)}{W(t)},
$$
then 
\begin{equation}\label{con el log por arriba}
\overline{W}(\la)\leq \bs{\w} \left(\log\dfrac 1\la \right)^{-1}, \qquad \, 0 < \la < 1.
\end{equation}

\medskip

\noi From here the following result follows  easily:

\begin{lemma}\label{better decay}
If $\w\in B^*_\infty$ then
$$
\overline{W}(\la) \leq e \la^{1/(e[\w]_{B^*_\infty})}, \qquad \, 0<\la<1. 
$$
\end{lemma} 
\bdem
Let $\la_0=e^{-e\bs{\w}}$. Since $\overline{W}$ is submultiplicative, by \eqref{con el log por arriba} and induction on $n \in \N \cup \{0\}$  we get that
$$
\overline{W}(\la_0^n)\leq \left( \overline{W}(\la_0) \right)^n \leq \left(\frac{1}{e}\right)^n=\Big(\la_0^{n}\Big)^{\frac 1{e \bs{\w}}}.
$$

Now take $\lambda \in (0,1)$ and choose $n\in \N \cup \{0\}$ such that $\lambda_0^{n+1} \leq \lambda < \lambda_0^n$. Then, since $\overline{W}$ is increasing,
$$\overline{W}(\lambda)\leq \overline{W}(\la_0^n)\leq \Big(\la_0^{n}\Big)^{\frac{1}{e\, \bs{\w}}} \leq e \la^{1/(e[\w]_{B^*_\infty})}.$$
\edem

\noi A similar result can be obtained for the $B^*_q$ weights. 

\begin{lemma}\label{better decay q}
Let $1 < q < \infty$. If $\w\in B^*_q$, then
\begin{equation*} 
\overline{W}(\la)\leq 4q\, \bsq{\w}{q} \la^{\frac{1}{q}+\frac 1{4q\bsq{\w}{q}}}.
\end{equation*}
\end{lemma} 

\begin{proof}

First of all, C.J. Neugebauer proved in \cite{n:n} that if $\w\in B^*_q$ for some $q\in (0,+\infty)$ then there exists some $\eps = \eps(q, \w) >0$ such that $\w\in B^*_{q-\eps}$. In particular, following the estimates used in  \cite{n:n}, for $q\geq 1$ and $\w\in B^*_q$, taking $\eps = \frac q{4[\w]_{B^*_q}}$ we have that
$$
\bsq{\w}{q-\eps}\leq 4 q\,\bsq{\w}{q}.
$$

\noindent Hence, from \eqref{eq 1 de Bp} we obtain that for every $0 < r < t$,
\begin{equation*}
\frac{W(r)}{W(t)} \leq \bsq{\w}{q-\eps} \Big(\frac rt\Big)^{\frac 1q}\Big(\frac rt\Big)^{\frac \varepsilon{q(q - \varepsilon)}}\leq 4q\, \bsq{\w}{q} \Big(\frac rt \Big)^{\frac 1q}\Big(\frac rt\Big)^{\frac 1{4q\bsq{\w}{q}}}.
\end{equation*}

\noindent Therefore, if $\la\in (0,1)$ we get that \begin{equation*}
\overline{W}(\la) \leq 4q\, \bsq{\w}{q}\la^{\frac{1}{q}+\frac 1{4q\bsq{\w}{q}}}.
\end{equation*}
\end{proof}

As an example of the weights presented we have the power weights, which will take an important role in the proofs of the main results.

\begin{lemma}\label{Bstar de potencias}
Let $\w(t)=t^{\tau-1}$.  

\medskip

\noindent
1)   If $0 < p < \infty$ and $0 < \tau \leq p $, then $\w\in B_p^{\cR}\cap B^*_\infty$ with 
$[\w]_{B_p^{\cR}}= 1$ and $ [\w]_{B^*_\infty}=\tau^{-1}$.

\medskip

\noindent
2) If $0 < q < \infty$ and  $\tau > \frac 1q$, then $\w\in B^*_q$ with $
[\w]_{B^*_q}=\dfrac \tau{\tau - \frac 1q}$. 

\end{lemma}

\section{Proof of our main results}\label{main}

\begin{proof}[Proof of the necessity of Theorem \ref{Teo-1}] We will first prove the result when $f=\chi_E$ with $E$ a measurable set of finite measure. Then, using our hypothesis 
with $\w(t)=t^{\tau-1}$ and Lemma \ref{Bstar de potencias},  we get that, for every $t > 0$ and every $\tau\in (0,1]$, 
\begin{equation}\label{pp1}
 (T\chi_E)^*(t) \lesssim \ \fii(\tau^{-1}) \left(\frac{|E|}{t}\right)^\tau. 
\end{equation}
Taking the infimum in $\tau \in (0,1]$, and using Lemma \ref{L infimo log log} we obtain that
\begin{align*}
    (T\chi_E)^*\left(t\right) & \lesssim \left[\left( \frac{|E|}t \right)\chi_{\left(|E|, \infty\right)}(t) 
    + \fii\left(1+\log \frac{|E|}{t} \right)\chi_{\left(0, |E|\right)}(t) \right], 
\end{align*}
and by Lemma~\ref{comparables}, 
\begin{align*}
    (T\chi_E)^*\left(t\right) & \lesssim \frac1t\int_0^t (\chi_E)^*(s)\, ds + \int_t^\infty \left(1+\log\frac st\right)^{- 1}\fii\left(1+\log\frac st \right) (\chi_E)^*(s) \, \frac{ds}s  
    \\&= S_{1,\infty,\fii}(\chi_E)^*(t).
\end{align*}

The extension to simple functions with compact support follows the same lines as the proof of Theorem III.4.7 of  \cite{bs:bs}. We include the computations adapted to our case for the sake of completeness. First of all, consider a positive simple function 
\begin{equation}\label{fsimplefunc}
f=\sum_{j=1}^na_j\chi_{F_j},
\end{equation}
where $F_1\subseteq F_2\subseteq \ldots\subseteq F_n$  have finite measure. Then
$$
f^*=\sum_{j=1}^na_j\chi_{[0,|F_j|)}.
$$
Using what we have already proved for characteristic functions we get 
\begin{equation}\label{Tf**estimationTeo2}\begin{split}
(Tf)^{**}(t)&\leq \sum_{j=1}^n a_j \left( T(\chi_{F_j})\right)^{**}(t)\lesssim \sum_{j=1}^n a_j \left( S_{1,\infty,\fii}(\chi_{[0,|F_j|)})\right)^{**}(t)\\
&=\left( S_{1,\infty,\fii}\left(\sum_{j=1}^n a_j \chi_{[0,|F_j|})\right)\right)^{**}(t)=  S_{1,\infty,\fii}(f^*)^{**}(t).
\end{split}
\end{equation}
Since $ S_{1,\infty,\fii}(f^*)^{**}= S_{1,\infty,\fii}(f^{**})$ (see Lemma~\ref{GeneralizedCalderonProperty}) we finally obtain that
\begin{equation}\label{T vs S doble star}
(Tf)^{**}(t)\lesssim  S_{1,\infty,\fii}(f^{**})(t).
\end{equation}

Fix $t>0$ and consider the set $E=\{x: f(x)>f^*(t)\}$. Using that set define
\begin{equation}\label{g_h_def}
g=(f-f^*(t))^+\chi_E \peso{and} h=f^*(t)\chi_E + f\chi_{E^c}.
\end{equation}
Then
$$
g^*(r)=(f^*(r)-f^*(t))^+  \peso{and}  h^*(r)=\min\{f^*(r),f^*(t)\}.
$$Since  $\w=1$ belong to $B_1^\cR\cap B^*_\infty$, the corresponding weak inequality leads to
$$
(Tg)^*(t/2)\lesssim \frac1t\int_0^t f^*(s)\,ds - f^*(t).
$$  
On the other hand, using \eqref{T vs S doble star} we get 
$$
(Th)^{**}(t)\lesssim S_{1,\infty,\fii}(h^{**})(t)=P_1(h^{**})(t)+Q_{\infty,\fii}(h^{**})(t)= f^*(t) +Q_{\infty,\fii}(h^{**})(t),
$$
where the last equality holds because $h^*(r)=f^*(t)$ for every $r\in[0,t]$. Now, consider the auxiliary function
$$
\widetilde{\fii}(x)=\frac{\fii(x)}x.
$$
By Fubini's theorem, 
\begin{equation}\label{Qinftyvarphih^**}
\begin{split}
Q_{\infty,\fii}(h^{**})(t) &=\int_t^\infty \widetilde{\fii} \left(1+\log \frac st\right) h^{**}(s) \, \frac{ds}s \\ 
&=\int_0^t f^*(t)\left(\int_t^{\infty} \widetilde\fii \left(1+\log \frac st\right) \frac{ds}{s^2 }\right)dr\\
&\quad +\int_t^\infty \left(\int_r^{\infty} \widetilde\fii\left(1+\log \frac st\right) \frac{ds}{s^2}\right) f^*(r)\,dr=I_1+I_2.
\end{split}
\end{equation}

On the one hand, the first integral is a multiple of $f^*(t)$. Indeed, 
\begin{align*}
I_1&=\int_0^t f^*(t)\left(\int_t^{\infty}\widetilde\fii\left(1+\log \frac st\right) \frac{ds}{s^2 }\right)dr
\\&=\frac{1}{t}\int_0^t f^*(t)\left(\int_1^{\infty} \widetilde\fii\left(1+\log u\right) \frac{du}{u^2 }\right)dr=C_1\, f^*(t).
\intertext{On the other hand, to study the second integral we will make use of Lemma \ref{comparables 2}. To do so, we take any $\la > 1$ and observe that}
I_2&=\int_t^\infty \left(\int_r^{\infty} \widetilde\fii\left(1+\log \frac st\right) \frac{ds}{s^2}\right) f^*(r)\,dr\\
&=\int_t^{\la t}+\int_{\la t}^\infty \left(\int_r^{\infty}\widetilde\fii\left(1+\log \frac st\right) \frac{ds}{s^2}\right) f^*(r)\, dr.
\end{align*}
The first part is similar to $I_1$, and it can be also controlled by a multiple of $f^*(t)$. Indeed, using that $f^*$ is decreasing we get 
\begin{align*}
\int_t^{\la t} \left(\int_r^{\infty} \widetilde\fii\left(1+\log \frac st\right) \frac{ds}{s^2}\right) f^*(r)\, dr&\leq \frac{f^*(t)}{t}\int_t^{\la t} \left(\int_{r/t}^{\infty}\widetilde\fii\left(1+\log u\right) \frac{du}{u^2}\right) dr\\
&\leq \frac{f^*(t)}{t}\int_t^{\la t} \left(\int_{1}^{\infty} \widetilde\fii\left(1+\log u\right) \frac{du}{u^2}\right) dr\\
&=C_2 \la\, f^*(t).
\end{align*}
For the second part, by  Lemma \ref{comparables 2} we know that there exists some $\la = \la(\varphi) > 1$ such that
$$
\int_r^{\infty}  \widetilde\fii\left(1+\log \frac st\right)  \frac{ds}{s^2}\approx    \frac{1}{r}\ \widetilde\fii \left(1+\log \frac rt\right) .
$$
Therefore,
\begin{align*}
\int_{\la t}^\infty \left(\int_r^{\infty} \widetilde\fii\left(1+\log \frac st\right) \frac{ds}{s^2}\right) f^*(r)\, dr \lesssim
\int_{\la t}^\infty \widetilde\fii\left(1+\log \frac rt\right)   f^*(r)\, \frac{dr}{r}.
\end{align*}

\medskip

In conclusion, putting $I_1$ and $I_2$ together we obtain that 
\begin{align*}
Q_{\infty,\varphi}(h^{**})(t) &\lesssim  f^*(t)+\int_t^\infty \widetilde\fii\left(1+\log \frac rt\right)   f^*(r)\,\frac{dr}{r}=f^*(t)+Q_{\infty,\fii} (f^*)(t).
\end{align*}
Thus,
\begin{align*}
(Tf)^*(t)& \leq  (Tg)^*(t/2) + (Th)^{**}(t/2) \lesssim S_{1,\infty,\varphi}(f^*)(t).
\end{align*}

Finally, the general case follows from this particular case dividing the function in its positive and negative parts.

\noindent \textit{Proof of the sufficiency of Theorem \ref{Teo-1}.} Suppose that $(Tf)^*(t) \lesssim S_{1,\infty,\varphi}(f^*)(t)$ for every $t > 0$.  The operator $S_{1,\infty,\fii}$ has the form
$$
S_{1,\infty,\fii}f(t)=\int_0^\infty k(t,s) f(s)ds,
$$
where the kernel is
$$
k(t, s)= \frac 1t \chi_{[0,t)}(s) + \frac{1}{s} \left(1+\log \frac s t \right)^{-1}\fii \left(1+\log \frac s t \right) \chi_{[t,\infty)} (s).
$$

\medskip

\noindent So, using Theorem 3.3 in \cite{cs:cs}, the norm $\norm{S_{1,\infty,\varphi}}_{\Lambda^1(w)\rightarrow\Lambda^{1,\infty}(w)}$ can be estimated by 
$$
A_k:=\sup_{t>0} \left( \sup_{r>0} \left( \int_0^r k(t,s) ds\right) W(r)^{-1} \right) W(t).
$$
Note that, if $0 < r<t$ then we have 
$$
\int_0^r k(t,s) \,ds= \frac rt.
$$
On the other hand, if $r>t$, by Lemma \ref{comparables} we obtain

$$
\int_0^r k(t,s)\, ds \approx  \fii \left(1+\log \frac r t \right). 
$$
As a consequence we have that
$$
A_k\approx \sup_{t>0} \max \left \{     \left( \sup_{0 < r<t} \frac{r}{t} \frac{W(t)}{W(r)}\right) , \left( \sup_{r>t}  
\fii \left(1+\log \frac r t \right)\frac{W(t)}{W(r)}\right)  \right\}.
$$
Since
$$
 \sup_{0 < r<t} \frac{r}{t} \frac{W(t)}{W(r)} \leq [\w]_{B_1^{\cR}},
$$
and $ [\w]_{B_1^{\cR}}\geq 1$ we get that
$$
A_k \lesssim  \max\left\{[\w]_{B_1^{\cR}}\  ,\  \sup_{t>0} \sup_{r>t} \left\{  \fii\left(1+\log \frac r t \right) \frac{W(t)}{W(r)}\right\}\right\}.
$$

\medskip

Further, if $\la=t/r < 1$, then by Corollary \ref{better decay}
$$
\frac{W(t)}{W(r)}=\frac{W(\la r)}{W(r)}\leq e  \la^{1/(e[\w]_{B^*_{\infty}})}.
$$
Therefore, 
\begin{align*}
\sup_{t>0} \sup_{r>t} \left\{  \fii \left(1+\log \frac r t \right) \frac{W(t)}{W(r)}\right\}
&\leq e \sup_{\lambda<1} \left\{  \fii  \left(1+\log  \frac{1}{\lambda} \right)     \la^{1/(e[\w]_{B^*_{\infty}})} \right\} \\
&\leq e^{1+1/e}\ \sup_{x>1} \left\{  \fii (x) e^{-x/(e [\w]_{B^*_{\infty}})} \right\}. 
\end{align*}
Finally, by Lemma \ref{L sup log log} and the inequality  \eqref{por la logconcavidad2} we obtain that 
\begin{align*}
\sup_{t>0} \sup_{r>t} &\left\{  \fii \left(1+\log \frac r t \right) \frac{W(t)}{W(r)}\right\}
\lesssim \fii (e[\w]_{B^*_{\infty}})\lesssim \fii ([\w]_{B^*_{\infty}}),
\end{align*}
and we arrive to the desired result
$$
A_k\lesssim \max\{[\w]_{B_1^{\cR}}\, ,\, \fii ([\w]_{B^*_{\infty}})\}.$$
\end{proof}

\begin{proof}[Proof of Theorem \ref{Teo-main1-bis}]
The proof follows by  direct application of Theorems \ref{main1} and \ref{Teo-1}. 
\end{proof}

\begin{proof}[Proof of Corollary \ref{fip}]  We see that the hypothesis is equivalent to equation \eqref{pp1}, and hence,  the proof follows immediately from  the proof of the necessity of Theorem \ref{Teo-1}.
\end{proof}

\begin{proof}[Proof of Theorem \ref{fip2}]    By the Rubio de Francia's extrapolation \cite{d:d}, the hypothesis implies that, for every $p\ge 1$, 
$$
||Tf||_{L^{p, \infty}}\lesssim \varphi(p) ||f||_{L^{p,1}}, 
$$
and the results follows from Corollary \ref{fip}.
\end{proof}

\medskip

Now, in order to prove Theorem \ref{Teo-main2-bis}, we first need to have the analogue of Theorem \ref{main1} which can be found in \cite{bc:bc2}.

\begin{theorem}\label{main2} 
Let $1 \leq p_0 < \infty$, $0 < \alpha \leq 1$ and let $T$ be an operator satisfying that, for every $u \in A_1$, $$T:L^{p_0,1}(u^\alpha) \rightarrow L^{p_0,\infty}(u^\alpha)$$ is bounded with constant less than or equal to $\varphi(||u||_{A_1})$, where $\varphi$ is an increasing function in $[1,\infty)$. Then, for every  $\w \in B_{\frac 1{p_0}}^{\cR}\cap B^*_{\frac{p_0}{1-\alpha}}$,
 $$
 T:\llouno\to\lloinuno
 $$
is bounded with norm less than or equal to $C_1[\w]_{B_{\frac 1{p_0}}^{\cR}}^{\frac 1{p_0}} \overline\varphi\parenthesis{C_2[\w]_{B_{\frac{p_0}{1-\alpha}}^*}}$, for some positive constants $C_1$, $C_2$ independent of $\w$ and where \begin{equation}\label{overlinevarphi}
    \overline{\varphi}(x) = \left\{\begin{array}{ll} \varphi(x^{\frac{p_0}\alpha}), & 0 < \alpha < 1, \vspace{3mm} \\ \varphi(x), & \alpha = 1. \end{array}\right.
\end{equation}
\end{theorem}

Further, we will need the following generalization of Theorem \ref{Teo-1}.

\begin{theorem}\label{Teo-2}
Let $T$ be a sublinear operator and let $\varphi$ be an admissible function. If for every exponent $1 \leq q_1 < q_2 \leq \infty$ and for every weight $\w \in B_{\frac 1{q_1}}^{\cR}\cap B^*_{q_2}$,$$T:\llouno\to\lloinuno$$ is bounded with norm less than or equal to $C[\w]_{B_{\frac 1{q_1}}^{\cR}}^{\frac 1{q_1}}\varphi\parenthesis{[\w]_{B_{q_2}^*}}$ then, for every $t>0$ and every measurable function $f$, 
\begin{equation*} 
\begin{split}
    (Tf)^*(t)& \lesssim S_{q_1,q_2,\varphi}(f^*)(t) \\ &:=  \frac 1{t^{\frac 1{q_1}}}\int_0^t f^*(s) \, \frac{ds}{s^{1 - \frac 1{q_1}}} + \left\{\begin{array}{ll} \displaystyle
        \frac 1{t^{\frac 1{q_2}}}\int_t^\infty \varphi\parenthesis{1 + \log \frac st}f^*(s)\, \frac{ds}{s^{1 - \frac 1{q_2}}}, & q_2 < \infty, \vspace{3mm} \\ \displaystyle
        \int_t^\infty \left(1 + \log \frac st \right)^{-1} \varphi\parenthesis{1 + \log \frac st}f^*(s)\, \frac{ds}s, & q_2 = \infty.
    \end{array}\right.
\end{split}
\end{equation*}

\medskip

\noi Conversely, suppose that $(Tf)^*(t) \lesssim S_{q_1,q_2,\varphi}(f^*)(t)$, $t > 0$. Then

\medskip

$$
\|T\|_{\llouno\to\lloinuno}\lesssim \begin{cases}
[\w]_{B_{\frac 1{q_1}}^{\cR}}^{\frac 1{q_1}}\ \bsq{\w}{q_2}\, \fii(\bsq{\w}{q_2}), & \mbox{$q_2<\infty$},\\ [\w]_{B_{\frac 1{q_1}}^{\cR}}^{\frac 1{q_1}}\, \fii(\bs{\w}), & \mbox{$q_2=\infty$}.
\end{cases}
$$

\end{theorem}

\medskip

\begin{proof}

First, assume that for every weight $\w \in B_{\frac 1{q_1}}^{\cR}\cap B^*_{q_2}$,$$T:\llouno\to\lloinuno$$ is bounded with norm less than or equal to $C[\w]_{B_{\frac 1{q_1}}^{\cR}}^{\frac 1{q_1}}\varphi\parenthesis{[\w]_{B_{q_2}^*}}$. Note that by Lemma~\ref{Bstar de potencias}, $\w(t)=t^{\tau - 1}$ belong to $B_{\frac 1{q_1}}^{\cR}\cap B^*_{q_2}$ for every $\tau \in \left(\frac 1{q_2}, \frac 1{q_1}\right]$. Hence, using our hypothesis, we obtain that for every measurable set $E$, $$ (T\chi_E)^*(t) \lesssim \varphi\parenthesis{\sqbracket{\tau - \frac1{q_2}}^{-1}}\left(\frac{|E|}{t}\right)^\tau = \left[\varphi\parenthesis{\tilde{\tau}^{-1}}\left(\frac{|E|}{t}\right)^{\tilde{\tau}}\right]\left(\frac{|E|}{t}\right)^{\frac 1{q_2}}, 
$$
with $\tilde{\tau} = \tau - \frac 1{q_2}$. Hence, taking the infimum in $\tilde{\tau} \in \left(0, \frac 1{q_1} - \frac 1{q_2}\right]$ and using Lemma \ref{L infimo log log} we get that
\begin{align*}
    (T\chi_E)^*\left(t\right) \lesssim \left( \frac{|E|}t \right)^{\frac 1{q_1}}\chi_{\left(|E|, \infty\right)}(t) + \varphi\left(1 + \log \frac{|E|}t \right)\left( \frac{|E|}t \right)^{\frac 1{q_2}}\chi_{\left(0, |E|\right)}(t).
\end{align*}
Then, by Lemma~\ref{comparables}, \begin{align*}
     (T\chi_E)^*\left(t\right) & \lesssim S_{q_1,q_2,\varphi}(\chi_E)^*(t).
\end{align*} 

The extension to positive simple functions with support in a set of finite measure follows the same lines as the proof of the necessity of Theorem~\ref{Teo-1} with few modifications. First of all, we consider a positive simple function like the one in \eqref{fsimplefunc}. Hence, as in \eqref{Tf**estimationTeo2}, using what we have already proved for characteristic functions together with the sublinearity of $T$ and the equivalence $S_{q_1,q_2,\varphi}(f^*)(t)^{**} \approx S_{q_1,q_2,\varphi}(f^{**})(t)$ (see Lemma~\ref{GeneralizedCalderonProperty}), we obtain that
\begin{equation}\label{T vs S_q_1,q_2 doble star}
(Tf)^{**}(t)\lesssim S_{q_1,q_2,\varphi}(f^{**})(t).
\end{equation}

So fix $t>0$ and take the functions $g$ and $h$ from $f$ defined in \eqref{g_h_def}. Since the weight $\w(r) = r^{\frac 1{q_1} - 1}$ is in $B_{\frac 1{q_1}}^\cR\cap B^*_{q_2}$, the corresponding weak inequality leads to
$$
(Tg)^*(t/2)\lesssim \frac 1{t^{\frac 1{q_1}}}\int_0^t (f^*(s)-f^*(t))\,\frac{ds}{s^{1 - \frac 1{q_1}}} \approx \frac1{t^{\frac 1{q_1}}}\int_0^t f^*(s)\,\frac{ds}{s^{1 - \frac 1{q_1}}} - f^*(t).
$$  
On the other hand, using \eqref{T vs S_q_1,q_2 doble star} for $h$ instead of $f$ we get
\begin{align*}
(Th)^{**}(t) &\lesssim S_{q_1,q_2, \varphi}(h^{**})(t) = P_{q_1}(h^{**})(t) + Q_{q_2, \varphi}(h^{**})(t).
\end{align*}

\noindent First, since $h^*(r) = f^*(t)$ for every $r \in [0,t]$, $$P_{q_1}(h^{**})(t) \approx f^*(t).$$

\noindent Besides, arguing as we did to bound \eqref{Qinftyvarphih^**} with the auxiliary function $$\widetilde\varphi(x) = \left\{\begin{array}{lc}
    \displaystyle \fii(x)e^{\frac{x-1}{q_2}}, & 1 < q_2 < \infty, \vspace{3mm} \\
    \displaystyle \frac{\fii(x)}x,\, & q_2 = \infty, 
\end{array}\right.$$ then we deduce that $$Q_{q_2, \varphi}(h^{**})(t) \lesssim f^*(t) + Q_{q_2, \varphi}(f^*)(t),$$ 
\noindent and the  result follows. 

\medskip

Conversely, assume that $(Tf)^*(t) \lesssim S_{q_1,q_2,\varphi}(f^*)(t)$ for every $t > 0$. If $q_2<\infty$ then the operator $S_{q_1,q_2,\varphi}$ has the form
$$
S_{q_1,q_2,\varphi}(f^*)(t) =\int_0^\infty k(t,s) f^*(s)ds,
$$
where the kernel is
$$
k(t, s)= \frac{1}{t^{\frac{1}{q_1}}} \chi_{[0,t)}(s)  \frac{1}{s^{1-\frac{1}{q_1}}}+ \fii \left(1+\log \frac s t \right) \left(\frac{s}{t}\right)^{\frac{1}{q_2}} \chi_{[t,\infty)} (s) \frac{1}{s}.
$$
\medskip

\noi By Theorem 3.3 in \cite{cs:cs}, the norm the norm $\norm{S_{q_1,q_2,\varphi}}_{\Lambda^1(w)\rightarrow\Lambda^{1,\infty}(w)}$ can be estimated by 
$$
A_k:=\sup_{t>0} \left( \sup_{r>0} \left( \int_0^r k(t,s) ds\right) W(r)^{-1} \right) W(t).
$$
Now if $0 < r<t$ then we have 
$$
\int_0^r k(t,s) \,ds= \left(\frac rt\right)^{\frac{1}{q_1}},
$$
while if $r>t$, then by Lemma \ref{comparables} we obtain
$$
\int_0^r k(t,s)\, ds \approx  \fii \left(1+\log \frac r t \right)  \left(\frac{r}{t}\right)^{\frac{1}{q_2}}. 
$$
In consequence, we have that
$$
A_k\approx \sup_{t>0} \max \left \{     \left( \sup_{0 < r<t} \left(\frac{r}{t}\right)^{\frac{1}{q_1}} \frac{W(t)}{W(r)}\right) , \left( \sup_{r>t}  
 \fii \left(1+\log \frac r t \right)  \left(\frac{r}{t}\right)^{\frac{1}{q_2}}\frac{W(t)}{W(r)}\right)  \right\}.
$$
Since
$$
 \sup_{0 <r<t} \left(\frac{r}{t}\right)^{\frac{1}{q_1}} \frac{W(t)}{W(r)} \leq [w]_{B_{\frac{1}{q_1}}^{\cR}}^{\frac{1}{q_1}},
$$ 
we get that
$$
A_k \lesssim  \max \left\{ [w]_{B_1^{\cR}}^{\frac{1}{q_1}}\ ,\  \sup_{t>0} \sup_{r>t} \left\{   \fii \left(1+\log \frac r t \right)  \left(\frac{r}{t}\right)^{\frac{1}{q_2}}\frac{W(t)}{W(r)} \right\}\right\}.
$$
Hence, if $\la=t/r < 1$, then by Lemma \ref{better decay q}
$$
\left(\frac{r}{t}\right)^{\frac{1}{q_2}} \frac{W(t)}{W(r)}=\la^{-\frac{1}{q_2}}\frac{W(\la r)}{W(r)}\leq 4q\, \bsq{\w}{q_2} \la^{\frac 1{4q_2\bsq{\w}{q_2}}}.
$$
Therefore
\begin{align*}
\sup_{t>0} \sup_{r>t} \left\{   \fii \left(1+\log \frac r t \right)  \left(\frac{r}{t}\right)^{\frac{1}{q_2}}\frac{W(t)}{W(r)} \right\}
&\lesssim  \bsq{\w}{q_2} \sup_{\lambda<1} \left\{  \fii  \left(1+\log  \frac{1}{\lambda} \right)    \la^{\frac 1{4q_2\bsq{\w}{q_2}}} \right\} \\
&\lesssim  \bsq{\w}{q_2} \sup_{x>1} \left\{  \fii (x) e^{-x/(4q_2\bsq{\w}{q_2})} \right\}. 
\end{align*}

Finally, by Lemma \ref{L sup log log} and the inequality  \eqref{por la logconcavidad2} we obtain that 

\begin{align*}
\sup_{t>0} \sup_{r>t} \left\{   \fii \left(1+\log \frac r t \right)  \left(\frac{r}{t}\right)^{\frac{1}{q_2}}\frac{W(t)}{W(r)} \right\}
&\lesssim \bsq{\w}{q_2} \fii \big(4q_2\bsq{\w}{q_2}\big)\\&\lesssim \bsq{\w}{q_2} \fii \big([w]_{B^*_{q_2}}\big).
\end{align*}
Combining both estimates we get
$$
A_k \lesssim  \max \Big\{ [w]_{B_1^{\cR}}^{\frac{1}{q_1}}\,,\,\bsq{\w}{q_2} \fii \big([w]_{B^*_{q_2}}\big) \Big\},
$$
which leads to the desired result. If $q_2=\infty$, the proof is a combination of the proof for the case $q_2<\infty$ and the proof of the sufficiency in Theorem \ref{Teo-1}.
\end{proof}

We are finally ready to prove our last main result:

\begin{proof}[Proof of Theorem \ref{Teo-main2-bis}]
Observe that $\overline{\varphi}$ as in \eqref{overlinevarphi} satisfies the same properties as $\varphi$. Therefore, the proof follows by direct application of Theorems \ref{main2} and \ref{Teo-2}. 
\end{proof}

\section{Examples and applications}\label{operators_bound}

We shall present several examples of very interesting operators in Harmonic analysis for which our results give a pointwise estimate of the decreasing rearrangement.

\subsection{Fefferman-Stein inequality}

An operator $T$ is said to satisfy a Fefferman-Stein's inequality (\cite{fs:fs}) if, for every positive and locally integrable function $u$, 
\begin{equation}\label{ffss}
\int_{\{ |Tf(x)|>y\}} u(x)dx \lesssim \int |f(x)| Mu(x) dx.
\end{equation}

Clearly, for every operator satisfying \eqref{ffss} we have that
$$
T:L^1(u) \longrightarrow L^{1,\infty}(u) 
$$
is bounded with norm less than or equal to  $C||u||_{A_1}$ and hence, as a consequence of Theorem \ref{Teo-main1-bis}, we get the following:

\begin{corollary}
For every $t > 0$ and every measurable function $f$,
$$
(Tf)^*(t) \lesssim  \frac 1t \int_0^t f^*(s) \, ds + \int_t^\infty   f^*(s) \, \frac{ds}s. 
$$
\end{corollary}

This is the case (among many others operators) of  the area function \cite{chw:chw} defined by
$$
Sf(x)=\left(\int_{|x-y|\le t} |\nabla_{y,t}  (f* P_t)(y)|^2\right)^{1/2}, 
$$
where
$$
\nabla_{y,t} =\left(\frac{\partial}{\partial y_1}, \frac{\partial}{\partial y_2}, \cdots, \frac{\partial}{\partial y_n}, \frac{\partial}{\partial t}\right), \qquad P_t(y)=\frac{c_nt}{(t^2+|y|^2)^{(n+1)/2}}. 
$$

\medskip

\subsection{Bochner-Riesz and Rough singular operators} To begin with, we  recall the definition of these operators. Then, we present a known 
quantitative inequality, which leads to pointwise estimations for the decreasing rearrangement of these operators by using our results.
 
\subsection*{Bochner-Riesz operators}
Let $$\hat{f}(\xi) = \int_{\mathbb R^n}f(x)e^{-2\pi i x\cdot \xi}\, dx, \hspace{5mm} \xi \in \mathbb R^n,$$ 
be the Fourier transform of  $f \in L^1\left( \mathbb R^n \right)$. Let $a_{+} = \max\{a,0\}$ denote the positive part of $a \in \mathbb R$. Given $\lambda > 0$,  the Bochner-Riesz operator $B_{\lambda}$ is defined by
 $$
 \widehat{B_{\lambda}f}(\xi) = \left( 1 - |\xi|^2 \right)^{\lambda}_{+}\hat{f}(\xi), \qquad \xi \in \mathbb R^n.
 $$

These operators were first introduced by Bochner \cite{b:b} and, since then, they have been widely studied. The case $\lambda = 0$ corresponds to the so-called disc multiplier, which is unbounded on $L^p(\mathbb R^n)$ if $n \geq 2$ and $p \ne 2$ \cite{f:f}. When $\lambda > \frac{n - 1}{2}$, it is known that $B_\lambda f$ is controlled by the Hardy-Littlewood maximal function $Mf$. As a consequence, all  weighted inequalities for $M$ are also satisfied  by $B_{\lambda}$. The value $\lambda = \frac{n - 1}{2}$ is called the critical index. In this case, Christ \cite{c:c} showed that $B_{\frac{n-1}2}$ is bounded from $L^1(\mathbb R^n)$ to $L^{1,\infty}(\mathbb R^n)$, and Vargas \cite{v:v} proved that $B_{\frac{n-1}2}$ is bounded from $L^1(u)$ to $L^{1,\infty}(u)$ for every $u \in A_1$.  
%Later, in \cite{lprr:lprr} the authors got to improve the bounding weight constant as follows:

\subsection*{Rough singular integrals}For $n > 1$, set $\Sigma^{n-1} = \left\{ x \in \mathbb{R}^n \, : \, |x| = 1 \right\}$ and let $\Omega$ be an homogeneous function of degree zero such that 
\begin{equation} \label{simm}
\int_{\Sigma^{n-1}}\Omega(x)\,dx = 0. 
\end{equation}
The rough singular integral operator is defined by  
 $$
 T_{\Omega}f(x) = \text{p.v.} \int_{\mathbb R^n}\frac{\Omega(y')}{|y|^n}f(x-y)\,dy, \qquad x \in \mathbb R^n,
 $$ 

\noindent with $y' = \frac{y}{|y|}$. This operator was first introduced by Calder\'on and Zygmund who proved that \cite{cz:cz, cz:cz2} $T_\Omega$ is bounded on $L^p$ if the even part of $\Omega$ belongs to $L \log_+ L$ and its odd part belongs to $L^1$. Since then, this operator has been widely studied \cite{c:c, cr:cr, v:v}. When $\Omega \in L^\infty$, Duoandikoetxea and Rubio de Francia \cite{JDJLRF} proved that, for $1 < p < \infty$, $$T_{\Omega}:L^p(u)\rightarrow L^p(u)$$ is bounded whenever $u \in A_p$, later improved in \cite{JD3}, \cite{DKW} and  \cite{lprr:lprr}.

\subsection*{Quantitative results and the pointwise estimation}

In \cite{lprr:lprr} the authors obtained the following quantitative result.

\begin{theorem} \label{perezetal}%\cite{lprr:lprr}
Let $T$ be either the Bochner-Riesz operator $B_{\frac{n-1}2}$, where $n > 1$ or the rough singular integral $T_\Omega$, where $\Omega$ is in $ L^{\infty}(\Sigma^{n-1})$ and satisfies \eqref{simm}.
Then, 
$$
T: L^2(u)\rightarrow L^{2,\infty}(u)
$$ 
is bounded with constant less than or equal to $C||u||^2_{A_2}$.
\end{theorem}

Hence, using Theorem \ref{fip2},  we have the following: 

\begin{corollary}
Consider the hypothesis of Theorem \ref{perezetal}. Then,  for every $t > 0$ and every  function $f$ such that $Tf$ is well defined, we have that 
$$
\left(T f\right)^*(t) \lesssim \frac 1t\int_0^tf^*(s)\,ds + \int_t^\infty \parenthesis{ 1 + \log \frac st }f^*(s)\,\frac{ds}s.
$$
\end{corollary}

\medskip
 
\subsection{Sparse Operators} These operators have become very popular due to its role in the so called  $A_2$ conjecture consisting in proving that if $T$ is a Calder\'on-Zygmund operator then
$$
||Tf||_{L^2(v)} \lesssim ||v||_{A_2}||f||_{L^2(v)}. 
$$
This result was first obtained by Hyt\"onen \cite{h:h} and then simplified by Lerner \cite{l:l, l:l2}, who proved that the norm of a Calder\'on-Zygmund operator in a Banach function space $\mathbb X$ is dominated by the supremum of the norm  in  $\mathbb X$ of all the possible sparse operators and then proved that every sparse operator is bounded on $L^2(v)$ for every weight $v \in A_2$ with sharp constant. Let us give the precise definition. A general dyadic grid $\mathcal D$ is a collection of cubes in $\mathbb R^n$ satisfying the following properties:

\begin{enumerate}
    \item[(i)] For any cube $Q \in \mathcal D$, its side length is $2^k$ for some $k \in \mathbb Z$.
    
    \item[(ii)] Every two cubes in $\mathcal D$ are either disjoint or one is wholly contained in the other.
    
    \item[(iii)] For every $k \in \mathbb Z$ and given $x \in \mathbb R^n$, there is only one cube in $\mathcal D$ of side length $2^k$ containing it.
\end{enumerate}

\noindent Let $0 < \eta < 1$, a collection of cubes $\mathcal{S}  \subset \mathcal D$ is called $\eta$-sparse if one can choose pairwise disjoint measurable sets $E_Q \subset Q$ with $|E_Q| \geq \eta |Q|$, where $Q\in\mathcal{S}$. So, given a $\eta$-sparse family of cubes $\mathcal{S}$, the sparse operator is defined by 
$$\mathcal{A}_Sf(x) = \sum_{Q\in \mathcal{S}}\left(\frac{1}{|Q|}\int_Qf(y)\,dy\right) \chi_Q(x), \qquad x \in \mathbb R^n.$$

Even though, weighted estimates for these operators are known, one can easily compute the norm in $L^p$ directly, using the standard duality technique: for every, $p\ge 1$ and $g\in L^{p',1 }$ with norm equal to 1, 
\begin{eqnarray*}
& &\int \mathcal{A}_\mathcal Sf(x)  g(x)dx = \sum_{Q\in \mathcal{S}}\left(\frac{1}{|Q|}\int_Qf(y)\,dy\right) \int_Q g(x)dx 
\\
&\leq& \frac 1\eta
\sum_{Q\in \mathcal{S}}\left(\frac{1}{|Q|}\int_Qf(y)\,dy\right) \frac 1{|Q|}\int_Q g(x)dx |E_Q| \le \frac 1\eta \int_{\mathbb R^n} Mf(x) Mg(x) dx
\\
&\le& \frac 1\eta ||Mf||_{L^{p, \infty}} ||Mg||_{L^{p', 1}} \le \frac{C_n}\eta \left(\frac{p'}{p'-1}\right) ||f||_{L^{p,1}}||g||_{L^{p',1 }} \approx p  ||f||_{L^{p,1}}. 
\end{eqnarray*}
Therefore, as a consequence of Corollary \ref{fip}, we get the following:

\begin{corollary}
For every $t > 0$ and every measurable function $f$,
$$
\left( \mathcal{A}_{\mathcal S}f \right)^*(t) \lesssim \frac 1t\int_0^tf^*(s)\,ds + \int_t^\infty f^*(s)\,\frac{ds}s.
$$
\end{corollary}

\medskip

\subsection{Fourier multipliers}

Given $m \in L^\infty(\mathbb R^n)$, we say that $T_m$ is a Fourier multiplier if, for every Schwartz function $f$, $$\widehat{T_mf}(\xi) = m(\xi) f(\xi),\qquad \xi \in \mathbb R^n,$$ and $m$ is called a multiplier. It has been of great interest to identify, when possible, for which maximal operators $\mathcal M$ the operator $T_m$ satisfies a Fefferman-Stein's type inequality in $L^2$ of the form $$\int_{\mathbb R^n} |T_mf(x)|^2u(x) \,dx \leq \int_{\mathbb R^n} |f(x)|^2\mathcal M u(x)\,dx,$$ for measurable functions $f$ and positive locally integrable functions $u$ (see for instance \cite{bJ:bJ, c:c2, cA:cA, cf:cf, sE:sE}).

In particular, we present the following interesting case. 

\begin{proposition}[\cite{bJ:bJ}]

\noindent If $m:\mathbb R \rightarrow \mathbb C$ is a bounded function so that
 \begin{equation}\label{variationondyadicint}
    \sup_{R > 0} \int_{R \leq |\xi| \leq 2R} |m'(\xi)|\,d\xi < \infty, 
\end{equation} 
then, for every measurable function $f$ and every positive locally integrable function  $u$, $$\int_\mathbb R |T_mf(x)|^2u(x)\,dx \leq C \int_\mathbb R |f(x)|^2M^7u(x)\,dx, $$ where $M^7 = M \underbrace{\circ \dots \circ}_7 M$ is the 7-fold composition of $M$ with itself.

\end{proposition}

Hence, if $m \in L^\infty(\mathbb R)$ satisfies \eqref{variationondyadicint}, then 
$$
T_m:L^2(u) \longrightarrow L^2(u), \qquad C\norm{u}_{A_1}^7.
$$ 
As a consequence of Theorem~\ref{Teo-main2-bis}, we obtain the following:

\begin{corollary}
For every $t > 0$ and every measurable function $f$,
$$
\left( T_m f \right)^*(t) \lesssim \frac 1{t^\frac 12} \int_0^tf^*(s)\,\frac{ds}{t^\frac 12} + \int_t^\infty \left(1 + \log\frac st\right)^6 f^*(s)\,\frac{ds}s.
$$
\end{corollary}

In this context of Fourier multipliers, let us now consider, for each $\gamma, \beta \in \mathbb R$,   the class $\mathcal C(\gamma, \beta)$ of bounded functions $m:\mathbb R\rightarrow \mathbb C$ for which $$\text{supp}(m) \subseteq \{\xi : |\xi|^\gamma \geq 1\}, \qquad \sup_{\xi \in \mathbb R} |\xi|^\beta |m(\xi)| < \infty,$$ and $$\sup_{R^\gamma \geq 1}\, \sup_{     I\subseteq [R, 2R], \, \ell(I) = R^{-\gamma + 1} } R^\beta \int_{\pm I}|m'(\xi)|\,d\xi < \infty.$$

\begin{proposition}[\cite{bJ:bJ}]

\noindent Let $\gamma, \beta \in \mathbb R$ such that $\gamma > 2\beta$. If $m \in \mathcal{C}(\gamma, \beta)$ then, for every measurable function $f$ and every positive locally integrable function  $u$, 
$$\int_\mathbb R |T_mf(x)|^2u(x)\,dx \leq C\int_\mathbb R |f(x)|^2M^6\left( \left( M^5u^{\frac \gamma{2\beta}} \right)^{\frac {2\beta}\gamma}\right)(x)\,dx.
$$

\end{proposition}

Therefore, under the hypotheses of the previous result, $$T_m:L^2(u^\frac {2\beta}\gamma) \longrightarrow L^2(u^\frac {2\beta}\gamma), \qquad C\left(\frac{\gamma}{\gamma - 2\beta}\right)^6\norm{u}_{A_1}^\frac {10\beta}\gamma,$$ so that, as a consequence of Theorem~\ref{Teo-main2-bis}, we have the following:

\begin{corollary}
For every $t > 0$ and every measurable function $f$,
$$
\left( T_m f \right)^*(t) \lesssim \frac 1{t^\frac 12} \int_0^tf^*(s)\,\frac{ds}{t^\frac 12} + \frac 1{t^{\frac{\gamma - 2\beta}{2\gamma}}}\int_t^\infty \left(1 + \log\frac st\right)^{10} f^*(s)\,\frac{ds}{s^{\frac{\gamma + 2\beta}{2\gamma}}}.
$$
\end{corollary}

\end{document}